\documentclass[11pt,reqno]{amsart}

\usepackage[a4paper,margin=1in]{geometry}

\usepackage{amsmath,amssymb,amsfonts,amsthm,mathtools}
\usepackage{graphicx}
\usepackage{tikz}
\usetikzlibrary{cd}
\usepackage{xcolor}
\usepackage{url}
\usepackage{microtype}
\usepackage{hyperref}

\numberwithin{equation}{section}
\numberwithin{table}{section}
\numberwithin{figure}{section}

\newcommand{\Z}{\mathbb{Z}}

\DeclareMathOperator{\con}{con}
\DeclareMathOperator{\num}{num}
\DeclareMathOperator{\F}{F}

\theoremstyle{plain}
\newtheorem{thm}{Theorem}[section]

\newtheorem{lemma}[thm]{Lemma}
\newtheorem{proposition}[thm]{Proposition}
\newtheorem{corollary}[thm]{Corollary}

\theoremstyle{definition}
\newtheorem{definition}[thm]{Definition}
\newtheorem{example}[thm]{Example}

\theoremstyle{remark}

\title[Inverses of Fibonacci and Lucas Numbers]
{Inverses of Fibonacci and Lucas Numbers via Rational Indices}

\author{Zekiye Pinar Cihan}
\address{Department of Mathematics,
Bilecik Seyh Edebali University,
Bilecik, Turkey}
\email{pinarcihan@icloud.com}

\author{Ilker Inam}
\address{Department of Mathematics,
Bilecik Seyh Edebali University,
Bilecik, Turkey}
\email{ilker.inam@gmail.com}
\email{ilker.inam@bilecik.edu.tr}



\begin{document}

\begin{abstract}
Fibonacci and Lucas numbers have been extensively studied in various algebraic and number-theoretic contexts, including their modular inverses and generalizations. Motivated by these developments, we study the inverses of Fibonacci and Lucas numbers with rational indices through the codenominator function $\F$. Uludağ and Gökmen (2022) showed that the rational-indexed Fibonacci number $F_X$, where $X\in\mathbb{Q}_{>0}$, can be expressed using $\F$, and that infinitely many such representations exist.

In this paper, we extend their work by deriving a general explicit formula for $F_X$ through the codenominator function. We establish precise conditions under which $F_X$ coincides with a Lucas number or deviates from the classical Fibonacci sequence. Moreover, by means of these generalized formulas, we compute the units of Fibonacci and Lucas numbers, thereby proving the existence of multiplicative inverses for all Fibonacci and Lucas numbers within this framework. These results reveal new structural properties of Fibonacci- and Lucas-related sequences and suggest further directions for number-theoretic exploration.
\end{abstract}

\maketitle

\section{Introduction}The arithmetic and structural properties of Fibonacci and Lucas sequences have long attracted attention across number theory. In particular, questions about divisibility, modular inverses, and generalized indexations of these classical sequences have produced a rich literature. Early work in this direction includes Duverney’s study of the irrationality of the sum of reciprocals of Fibonacci numbers \cite{duverney1997irrationalite}, while more recent contributions have examined modular and multiplicative inverses for Fibonacci-type sequences \cite{sanna2023inverse,song2019modular}. Parallel matrix-based investigations study determinants and generalized inverses of matrices built from Fibonacci and Lucas entries; see, for example \cite{shen2011determinants, shen2016moore}.The continued fraction is also playing a significant role in our formulas, which can be used in the many areas of number theory problems. For example, \cite{benjamin2005pythagorean} is an interesting paper which proves that every prime of the form $4m+1$ is the sum of the squares of two positive integers.The conumerator and the codenominator functions are other players we use for formulas of rational-indexed Fibonacci and Lucas numbers. Some significant papers where they are defined are: \cite{ULUDAG2022103192}, \cite{uludaug2022quantizations}. A different perspective was introduced by Uludağ and Gökmen \cite{ULUDAG2022103192}, who defined the conumerator and codenominator functions and used them to represent Fibonacci numbers at rational indices. Their framework shows that rational-indexed Fibonacci values can be expressed by the codenominator function $\F$, thereby providing a natural way to extend the classical sequence beyond integer indices.Main contributions of the present paper are the following:Motivated by their approach, we ask whether one can construct formulas for $F_X$ for arbitrary rational indices $X$, whether we can determine when such values coincide with classical Fibonacci or Lucas numbers, and whether the inverses of Fibonacci and Lucas numbers can also be described in this setting. In this paper we provide partial answers to these questions. More precisely, we derive explicit formulas for rational-indexed Fibonacci numbers $F_X$, where the continued fraction expansion of $X \in \mathbb{Q}_{>0}$ has length up to six. Using these formulas we (i) characterize when a rational-indexed Fibonacci number coincides with a Lucas number, (ii) identify when it diverges from the integer-indexed Fibonacci sequence, and (iii) explicitly prove that every Fibonacci and Lucas number has a multiplicative inverse via formulas we find with the common denominator function $\F$.
Our results both extend the rational-indexed viewpoint of Uludağ and Gökmen and open the way to further investigations of arithmetic properties of Fibonacci- and Lucas-related sequences under rational indexation.The paper is organized as follows. In Section~\ref{sec:1} we recall the conumerator/codenominator framework, fix notation, and collect required identities for Fibonacci and Lucas numbers. Section~\ref{sec:2} derives closed formulas for rational-indexed Fibonacci and Lucas numbers for continued fractions of small length. In Section~\ref{sec:3} we apply these formulas to compute inverses and units and state the main theorems and corollaries. We conclude with remarks on possible extensions and directions for future work.

\section{Preliminaries}\label{sec:1}\subsection{The codenominator function $F$}First of all, we define the numerator function:
\begin{definition}[Numerator Function] {\normalfont(\cite{ULUDAG2022103192})}For $X = \tfrac{p}{q}\in\mathbb{Q}_{>0}$ with $\gcd (p, q) = 1$, we define the \emph{numerator function}	\begin{equation*}		\num\!\left(\tfrac{p}{q}\right) = p.	\end{equation*}	Thus, $\num(p/q)$ simply returns the numerator of $X$ in the lowest terms.\end{definition}

It can be easily seen that it satisfies the following functional equations:\begin{equation}	\num(1) = 1,\end{equation}\begin{equation}\num(1+X) =\num(X) + \mathrm{num}\!\left(\tfrac{1}{X}\right),\end{equation}\begin{equation}	\mathrm{num}\!\left(\tfrac{X}{X+1}\right) =\mathrm{num}(X),\end{equation}\begin{equation}	\mathrm{num}(X) = X \, \mathrm{num}(1/X).\end{equation}

Secondly, we define the conumerator function.\begin{definition}[The conumerator function]{\normalfont(\cite{ULUDAG2022103192})}	The \emph{conumerator function} is the map	\begin{equation*}		\mathrm{con}: \mathbb{Q}_{>0} \longrightarrow \mathbb{Z}_{>0}	\end{equation*}	determined recursively by	\begin{align}\label{con:1}	\mathrm{con}(1+X) &:= \con(X) + \con(1/X),\\		\mathrm{con}\!\left(\tfrac{1}{1+X}\right) &:= \con(X),\label{con:2}	\end{align}	with initial condition $\con(1) := 1$. 	This function is uniquely defined and will serve as a companion to the numerator function.\end{definition}

Now, we define the codenominator function $\F$, which will play an important role in our assertions.\begin{definition}[Codenominator function]{\normalfont((\cite{ULUDAG2022103192})}	The \emph{codenominator function} is defined in terms of the conumerator by	\begin{equation}		\F(X) := \con(1/X), \qquad X\in \mathbb{Q}_{>0}.	\end{equation}	It satisfies 	\begin{align}		\F(1+1/X) &= \F(X) \iff \F(1+X) = \F(1/X) = \con(X),\label{F:1}\\		\F\!\left(\tfrac{1}{1+X}\right) &= \F(X) + \F(1/X),\label{F:2}\\		\F(1)& = 1,	\end{align}	and will play a central role in expressing rational-indexed Fibonacci and Lucas numbers.\end{definition}One can see that the relation between the codenominator function and the Fibonacci numbers is as follows:\begin{align}	\F(n)& = \F(1+1/n),\\	\F(1+n) & = \F(1/n),\\	\F(2+n) &= \F(1+(1+n)) =\F\!\left(\tfrac{1}{1+n}\right) = \F(n) + \F(1/n) = \F(n) + \F(n+1),\end{align}where $n$ is a positive integer. Thus, the codenominator function $\F$ extends the Fibonacci sequence to $\mathbb{Q}_{>0}$:\begin{equation}	\F(n) := F_n,\end{equation}where $F_n$ is the usual Fibonacci sequence for $n \in \mathbb{Z}_{>0}$ with $F_0 = 0$, $F_1=1$, defined by\begin{equation}	F_{n} = F_{n-2} + F_{n-1}.\end{equation}Hence we can write\begin{equation}	\F(n+2) := F_n + F_{n+1} = F_{n+2}.\end{equation}For the details on the codenominator function $\F$, see \cite{ULUDAG2022103192}.\subsection{Fibonacci and Lucas Sequences}\subsubsection{The Fibonacci Sequence}Fibonacci numbers are a well-known sequence that originated from Leonardo Fibonacci's observation of the rabbit population. The recursive definition of the Fibonacci sequence was written down by Johannes Kepler in the 16th century (\cite{koshy2017fibonacci}).\begin{definition}	Let $\{F\}_{n\geq 0}$ be the Fibonacci sequence with $F_0 = 0$, $F_1 = F_2 = 1$. The recursive definition of the Fibonacci sequence is	\begin{equation}		F_n = F_{n-1} + F_{n-2}, \quad n\geq 2.	\end{equation}\end{definition}The following theorem gives the Binet formula for the Fibonacci sequence, found by the French mathematician Jacques-Marie Binet in 1843 (see \cite[p.~90]{koshy2019fibonacci}).
\begin{thm}[Binet's Formula]\label{b:1}{\normalfont(\cite{koshy2017fibonacci})}	Let $\alpha$ be the positive root of $x^2 - x - 1 = 0$ and $\beta$ the negative root. Then 	\begin{equation}		F_n = \frac{\alpha^n-\beta^n}{\alpha - \beta}, \qquad \alpha = \tfrac{1+\sqrt{5}}{2}, \quad \beta = \tfrac{1-\sqrt{5}}{2},	\end{equation}	for $n\geq 1$.
\end{thm}

This theorem is very useful for computing Fibonacci numbers explicitly.\subsubsection{The Lucas Sequence}The Lucas sequence was constructed with different initial conditions but the same recurrence as the Fibonacci sequence.\begin{definition}	Let $\{L\}_{n\geq 0}$ be the Lucas sequence with $L_0 = 2$, $L_1 = 1$. The recursive definition is	\begin{equation}		L_n = L_{n-1} + L_{n-2}, \qquad n\geq 2.	\end{equation}	Another recurrence is	\begin{equation}		L_n = F_{n-1} + F_{n+1}, \qquad n\geq 1.	\end{equation}\end{definition}\begin{thm}[Binet's Formula]{\normalfont(\cite{koshy2017fibonacci})}\label{b:2}	For $n\geq 0$, the Binet formula for the Lucas sequence is	\begin{equation}		L_n = \alpha^n + \beta^n, \qquad \alpha = \tfrac{1+\sqrt{5}}{2}, \quad \beta = \tfrac{1-\sqrt{5}}{2}.	\end{equation}\end{thm}The Fibonacci and Lucas sequences with negative indices are given by\begin{align}	F_{-n} &= (-1)^{n-1}F_n, \label{nf:1}\\	L_{-n} &= (-1)^nL_n. \end{align}These formulas appear in \cite[p.~8]{koshy2019fibonacci}.	\section{ Fibonacci and Lucas Sequence with Rational Indices }\label{sec:2}	\subsection{Rational-Indexed Fibonacci Numbers}		The following lemmas describe how rational indices of Fibonacci numbers can be expressed in terms of the codenominator function $\F$.		\begin{lemma}\label{l:2}		Let $X = [n_0; n_1] = n_0 + \tfrac{1}{n_1}$, where $n_0, n_1$ are positive integers. 		Let $\F(\cdot)$ denote the codenominator function, and let $F_n$ denote the Fibonacci sequence. 		Then		\begin{equation}			\F([n_0; n_1]) = F_{n_0}F_{n_1} + F_{n_0-1}F_{n_1+1}.		\end{equation}	\end{lemma}		\begin{proof}		Let $X = \tfrac{p}{q}$ with $q \neq 0$, which can be expressed as the continued fraction $X = [n_0; n_1] = n_0 + \tfrac{1}{n_1}$. 		We compute the codenominator $\F(X)$ using equations~\eqref{F:1} and~\eqref{F:2}:		\begin{align*}			\F([n_0; n_1]) 			&= \F\!\left(n_0+\tfrac{1}{n_1}\right) 			= \F\!\left(1+\tfrac{(n_0-1)n_1+1}{n_1}\right) \\			&= \F\!\left(\tfrac{1}{1+\tfrac{(n_0-3)n_1+1}{n_1}}\right) 			+ \F\!\left(1+\tfrac{(n_0-3)n_1+1}{n_1}\right)\\			& = 2\F\!\left(\frac{1}{1+\frac{(n_0-4)n_1+1}{n_1}}\right) + \F\!\left(1+\frac{(n_0-4)n_1+1}{n_1}\right)\\			&=F_3\F\!\left(\frac{1}{1+\frac{(n_0-4)n_1+1}{n_1}}\right) + F_2 \F\!\left(1+\frac{(n_0-4)n_1+1}{n_1}\right),		\end{align*}		since $F_2 =1$, $F_3 = 2$. Iterating this process and repeatedly applying equations~\eqref{F:1} and~\eqref{F:2}, we obtain		\begin{align*}			\F([n_0; n_1]) 			&= F_{n_0-1} \, \F\!\left(\tfrac{1}{1+\tfrac{1}{n_1}}\right) 			+ F_{n_0-2} \, \F\!\left(1+\tfrac{1}{n_1}\right) \\			&= (F_{n_0-1}+F_{n_0-2})\F(n_1) + F_{n_0-1}\F(n_1+1) \\			&= F_{n_0}F_{n_1} + F_{n_0-1}F_{n_1+1},		\end{align*}		where we have used the recurrence $F_n = F_{n-1}+F_{n-2}$ and the relation $\F(n) = F_n$.	\end{proof}		\begin{lemma}\label{l:4}		Let $X=[n_0;n_1,n_2]=n_0+\tfrac{1}{n_1+\tfrac{1}{n_2}}$ with $n_0,n_1,n_2\in\mathbb{Z}_{>0}$. 		Let $\F(\cdot)$ denote the codenominator function, $F_n$ the Fibonacci numbers and $L_n$ the Lucas numbers. Then		\[		\F([n_0;n_1,n_2])		= (F_{n_1-1}L_{n_0} + F_{n_1-2}F_{n_0+1})\,F_{n_2}		+ (F_{n_1-2}L_{n_0} + F_{n_1-3}F_{n_0+1})\,F_{n_2+1}.		\]	\end{lemma}		\begin{proof}		Let $X_1 = \frac{p_1}{q_1}$ for $q_1\neq 0$, and it can be written as the continued fraction \begin{equation*} X_1 = [n_0; n_1, n_2] = n_0 + \frac{1}{n_1 + \frac{1}{n_2}}.\end{equation*} 					Compute $\F([n_0;n_1,n_2])$ by iteratively applying the defining relations~\eqref{F:1} and~\eqref{F:2} of the codenominator function to the continued fraction expansion. Grouping terms after $n_0-1$ iterations yields		\[		\F([n_0;n_1,n_2]) = F_{n_0-1}\F\!\Big(\frac{1}{1+\frac{n_2}{n_1n_2+1}}\Big)		+ F_{n_0-2}\F\!\Big(1+\frac{n_2}{n_1n_2+1}\Big).		\]		Now apply Lemma~\ref{l:2} to the two codenominator values on the right-hand side, which are continued fractions with leading index $1$ and denominator depending on $n_1,n_2$. Simplifying with the Fibonacci recurrence and using $L_{n}=F_{n+1}+F_{n-1}$ yields the stated formula, i.e.;	\begin{align*}			\F([n_0; n_1, n_2]) &= (F_{n_1-2}L_{n_0}+F_{n_1-3}F_{n_0+1})\F(\frac{1}{1+\frac{(n_1-n_1)n_2+1}{n_2}})\\ &+(F_{n_1-3}L_{n_0} + F_{n_1-4}F_{n_0+1})\F(1+\frac{(n_1-n_1)n_2+1}{n_2})\\			&= (F_{n_1-2}L_{n_0}+F_{n_1-3}F_{n_0+1})\F(\frac{1}{1+\frac{1}{n_2}})+(F_{n_1-3}L_{n_0} + F_{n_1-4}F_{n_0+1})\F(1+\frac{1}{n_2})\\			&= (F_{n_1-1}L_{n_0} + F_{n_1-2}F_{n_0+1})\F(n_2) + (F_{n_1-2}L_{n_0}+ F_{n_1-3}F_{n_0+1})\F(1+n_2)\\			& = (F_{n_1-1}L_{n_0} + F_{n_1-2}F_{n_0+1})F_{n_2} + (F_{n_1-2}L_{n_0}+ F_{n_1-3}F_{n_0+1})F_{n_2+1},		\end{align*}		since $F_{n_1-1} = F_{n_1-2} + F_{n_1-3}$ and $F_{n_1-2} = F_{n_1-3} + F_{n_1-4}$.	\end{proof}	\begin{lemma}\label{l:6}		Let $X=[n_0;n_1,n_2,n_3]$ with $n_i\in\mathbb{Z}_{>0}$. Let $\F(\cdot)$ denote the codenominator function, let $F_n$ denote the Fibonacci sequence, and let $L_n$ denote the Lucas sequence.		Then the rational-indexed Fibonacci number $F_X$ is given by		\begin{align*}			F_X &= \Big((F_{n_2-1}L_{n_1} + F_{n_2-2}F_{n_1+1})L_{n_0} 			+ (F_{n_2-1}L_{n_1-1} + F_{n_2-2}F_{n_1})F_{n_0+1}\Big)F_{n_3} \\			&\quad+ \Big((F_{n_2-2}L_{n_1} + F_{n_2-3}F_{n_1+1})L_{n_0} 			+ (F_{n_2-2}L_{n_1-1} + F_{n_2-3}F_{n_1})F_{n_0+1}\Big)F_{n_3+1}.		\end{align*}	\end{lemma}		\begin{proof}		We use properties of the codenominator function $\F$, namely equations (\ref{F:1}) and (\ref{F:2}). The proof is based on the same method as in Lemma \ref{l:4}:		\begin{equation*}			\F(X) = \F(n_0 + \frac{1}{n_1+\frac{1}{n_2+\frac{1}{n_3}}}) = \F(1 + \frac{(n_0-1)k + t}{k}),		\end{equation*}		where $k = n_1n_2n_3 + n_1 + n_3$, and $t = n_2n_3 + 1$. Using Lemma (\ref{l:2}), we get		\begin{align*}			\F(X) &= F_{n_0-1}\F(\frac{1}{1 + \frac{t}{k}}) + F_{n_0-2}\F(1 + \frac{t}{k})\\			&= F_{n_0-1}\F(\frac{1}{1+\frac{n_2n_3+1}{n_1n_2 n_3+ n_1 + n_3}}) + F_{n_0-2}\F(1+\frac{n_2n_3+1}{n_1n_2n_3+n_1+n_3}) \\			&= F_{n_0}\F(1+\frac{(n_1-1)t + n_3}{t}) + F_{n_0-1}\F(\frac{1}{1+\frac{(n_1-1)t + n_3}{t}}).		\end{align*}		Using Lemma (\ref{l:4}), we obtain	\[	\F(X) = \big(F_{n_1-1}L_{n_0} + F_{n_1-2}F_{n_0+1}\big)	\F\!\left(\frac{1}{1+\tfrac{n_3}{t}}\right)	+ \big(F_{n_1-2}L_{n_0} + F_{n_1-3}F_{n_0+1}\big)	\F\!\left(1+\tfrac{n_3}{t}\right).	\]		Simplifying, we get 	\[	\F(X) = \big(F_{n_1-1}L_{n_0} + F_{n_1-2}F_{n_0+1}\big)	\F\!\left(\frac{1}{1+\tfrac{(n_2-1)n_3+1}{n_3}}\right)	+ \big(F_{n_1}L_{n_0} + F_{n_1-1}F_{n_0+1}\big)	\F\!\left(1+\tfrac{(n_2-1)n_3+1}{n_3}\right),	\]	and hence		\[		\F(X) = \big(F_{n_1+1}L_{n_0} + F_{n_1}F_{n_0+1}\big)		\F\!\left(\frac{1}{1+\tfrac{(n_2-2)n_3+1}{n_3}}\right)		+ \big(F_{n_1-1}L_{n_0} + F_{n_1-2}F_{n_0+1}\big)		\F\!\left(1+\tfrac{(n_2-2)n_3+1}{n_3}\right).		\]		Now, using Lemmas (\ref{l:2}) and (\ref{l:4}), we deduce		\[		\begin{aligned}			\F(X) &=			\Big( (F_{n_2-2}L_{n_1} + F_{n_2-3}F_{n_1+1}) L_{n_0}			+ (F_{n_2-2}L_{n_1-1} + F_{n_2-3}F_{n_1}) F_{n_0} \Big)			\F\!\left(\tfrac{1}{1+\tfrac{1}{n_3}}\right) \\[6pt]			&\quad +			\Big( (F_{n_2-3}L_{n_1} + F_{n_2-4}F_{n_1+1}) L_{n_0}			+ (F_{n_2-3}L_{n_1-1} + F_{n_2-4}F_{n_1}) F_{n_0} \Big)			\F\!\left(1+\tfrac{1}{n_3}\right).		\end{aligned}		\]		Finally, applying the properties of the codenominator function $F$, we get 		\[		\begin{aligned}			\F(X) = F_X &=			\Big((F_{n_2-1}L_{n_1} + F_{n_2-2}F_{n_1+1}) L_{n_0}			+ (F_{n_2-1}L_{n_1-1} + F_{n_2-2}F_{n_1}) F_{n_0+1}\Big) F_{n_3} \\[6pt]			&\quad +			\Big((F_{n_2-2}L_{n_1} + F_{n_2-3}F_{n_1+1}) L_{n_0}			+ (F_{n_2-2}L_{n_1-1} + F_{n_2-3}F_{n_1}) F_{n_0+1}\Big) F_{n_3+1}.		\end{aligned}		\]		Thus the proof is completed.	\end{proof}	\begin{lemma}\label{l:7}			Let $X=[n_0;n_1,n_2,n_3, n_4]$ with $n_i\in\mathbb{Z}_{>0}$. Let $\F(\cdot)$ denote the codenominator function, let $F_n$ denote the Fibonacci sequence, and let $L_n$ denote the Lucas sequence. 		Then the rational-indexed Fibonacci number $F_X$ is given by			\begin{align*}				F_X :&= (((F_{n_3-2}F_{n_2-1} + F_{n_3-1}L_{n_2-2})F_{n_1} + (F_{n_3-2}F_{n_2} + F_{n_3-1}L_{n_2-1})L_{n_1-1})L_{n_0}\\				& + ((F_{n_3-2}F_{n_2-1} + F_{n_3-1}L_{n_2-2})F_{n_1-1} + (F_{n_3-2}F_{n_2} + F_{n_3-1}L_{n_2-1})L_{n_1-2})F_{n_0+1})F_{n_4}\\				&+(((F_{n_3-3}F_{n_2-1} + F_{n_3-2}L_{n_2-2})F_{n_1} + (F_{n_3-3}F_{n_2} + F_{n_3-2}L_{n_2-1})L_{n_1-1})L_{n_0}\\				& + ((F_{n_3-3}F_{n_2-1} + F_{n_3-2}L_{n_2-2})F_{n_1-1} + (F_{n_3-3}F_{n_2} + F_{n_3-2}L_{n_2-1})L_{n_1-2})F_{n_0+1})F_{n_4+1}.			\end{align*}	\end{lemma}	\begin{proof}	The proof is the same as the proof of Lemma \ref{l:6}. First of all, we define the codenominator function $\F(X)$:		\begin{equation*}			\F(X) = \F(n_0 + \frac{1}{n_1+\frac{1}{n_2+\frac{1}{n_3+\frac{1}{n_4}}}}) = \F(1+\frac{(n_0-1)y + z}{y}),		\end{equation*}		where $y = n_1n_2n_3n_4 + n_1n_3 + n_1n_4 + n_3n_4 + 1$, and $z = n_2n_3n_4 + n_3 + n_4$. Now applying Lemma \ref{l:6} here, we deduce		\begin{align*}				\F(X) &= ((F_{n_2-2}L_{n_1-1} + F_{n_2-3}F_{n_1})L_{n_0} + (F_{n_2-2}L_{n_1-2} + F_{n_2-3}F_{n_1-1})F_{n_0+1})\F(\frac{1}{1+\frac{n_4}{n_3n_4+1}}) \\			&+ ((F_{n_2-3}L_{n_1-1} + F_{n_2-4}F_{n_1})L_{n_0} + (F_{n_2-3}L_{n_1-2} + F_{n_2-4}F_{n_1-1})F_{n_0+1})\F(1+\frac{n_4}{n_3n_4+1})\\			&= ((F_{n_2-1}L_{n_1-1} + F_{n_2-2}F_{n_1})L_{n_0} + (F_{n_2-1}L_{n_1-2} + F_{n_2-2}F_{n_1-1})F_{n_0+1})\F(1 + \frac{(n_3-1)n_4+1}{n_4})\\			&+ ((F_{n_2-2}L_{n_1-1} + F_{n_2-3}F_{n_1})L_{n_0} + (F_{n_2-2}L_{n_1-2} + F_{n_2-3}F_{n_1-1})F_{n_0+1})\F(\frac{1}{1+\frac{(n_3-1)n_4+1}{n_4}})\\			&= ((F_{n_2}L_{n_1-1} + F_{n_2-1}F_{n_1})L_{n_0} + (F_{n_2}L_{n_1-2} + F_{n_2-1}F_{n_1-1})F_{n_0+1})\F( \frac{1}{1+\frac{(n_3-2)n_4+1}{n_4}})\\			&+ ((F_{n_2-2}L_{n_1-1} + F_{n_2-3}F_{n_1})L_{n_0} + (F_{n_2-2}L_{n_1-2} + F_{n_2-3}F_{n_1-1})F_{n_0+1})\F(1+\frac{(n_3-2)n_4+1}{n_4}).		\end{align*}		Using the properties of codenominator function $\F$, it follows that		\begin{align*}			\F(X) &= (((F_{n_3-3}F_{n_2-1} + F_{n_3-2}L_{n_2-2})F_{n_1} + (F_{n_3-3}F_{n_2} + F_{n_3-2}L_{n_2-1})L_{n_1-1})L_{n_0}\\			& + ((F_{n_3-3}F_{n_2-1} + F_{n_3-2}L_{n_2-2})F_{n_1-1} + (F_{n_3-3}F_{n_2} + F_{n_3-2}L_{n_2-1})L_{n_1-2})F_{n_0+1})\F(\frac{1}{1+\frac{1}{n_4}})\\			&+(((F_{n_3-4}F_{n_2-1} + F_{n_3-3}L_{n_2-2})F_{n_1} + (F_{n_3-4}F_{n_2} + F_{n_3-3}L_{n_2-1})L_{n_1-1})L_{n_0}\\			& + ((F_{n_3-4}F_{n_2-1} + F_{n_3-3}L_{n_2-2})F_{n_1-1} + (F_{n_3-4}F_{n_2} + F_{n_3-3}L_{n_2-1})L_{n_1-2})F_{n_0+1})\F(1+\frac{1}{n_4}).		\end{align*}		Thus we get		\begin{align*}			\F(X) = F_X &= (((F_{n_3-2}F_{n_2-1} + F_{n_3-1}L_{n_2-2})F_{n_1} + (F_{n_3-2}F_{n_2} + F_{n_3-1}L_{n_2-1})L_{n_1-1})L_{n_0}\\			& + ((F_{n_3-2}F_{n_2-1} + F_{n_3-1}L_{n_2-2})F_{n_1-1} + (F_{n_3-2}F_{n_2} + F_{n_3-1}L_{n_2-1})L_{n_1-2})F_{n_0+1})F_{n_4}\\			&+(((F_{n_3-3}F_{n_2-1} + F_{n_3-2}L_{n_2-2})F_{n_1} + (F_{n_3-3}F_{n_2} + F_{n_3-2}L_{n_2-1})L_{n_1-1})L_{n_0}\\			& + ((F_{n_3-3}F_{n_2-1} + F_{n_3-2}L_{n_2-2})F_{n_1-1} + (F_{n_3-3}F_{n_2} + F_{n_3-2}L_{n_2-1})L_{n_1-2})F_{n_0+1})F_{n_4+1}.		\end{align*}		The proof has been already completed by the last equation.	\end{proof}		
	\section{The General Formula of The Rational-Indexed Fibonacci Numbers}Next lemma is written as a result of \cite[prop.~4]{ULUDAG2022103192}
	\begin{lemma}	\begin{equation*}		\F([n_0; n_1, \cdots, n_k]) = F_{n_0}\F([n_1; n_2,\cdots, n_k]) + F_{n_0-1}\F([n_1+1; n_2,\cdots, n_k]).	\end{equation*}\end{lemma}We know that	\begin{equation*}		\F([n_0; n_1]) = F_{n_0}F_{n_1} + F_{n_0-1}F_{n_1+1} = F^{(n_1)},	\end{equation*}	\begin{equation*}		\F([n_0; n_1-1]) = F_{n_0}F_{n_1-1} + F_{n_0-1}F_{n_1} = F^{(n_1-1)}.	\end{equation*}	The proof can be completed by induction.	\begin{lemma}		\begin{equation*}		\F([n_0; n_1,\cdots, n_k]) = F^{(n_{k-1})}\cdot F_{n_k} + F^{(n_{k-1}-1)}\cdot F_{n_k+1},		\end{equation*}		for $n_i\in\Z_+$.	\end{lemma}		\begin{proof}		\begin{align*}				\F([n_0; n_1, \cdots, n_k]) &= F_{n_0}\F([n_1; n_2,\cdots, n_k]) + F_{n_0-1}\F([n_1+1; n_2,\cdots, n_k])\\				=&F_{n_0}(F_{n_1}\F([n_2; n_3,\cdots, n_k]) + F_{n_1-1}\F([n_2+1; n_3,\cdots, n_k]))\\				+& F_{n_0-1}(F_{n_1+1}\F([n_2; n_3,\cdots, n_k]) + F_{n_1}\F([n_2+1; n_3,\cdots, n_k]))\\				=& (F_{n_0}F_{n_1} + F_{n_0-1}F_{n_1+1})\F([n_2; n_3,\cdots, n_k]) \\				+& ( F_{n_0}F_{n_1-1} + F_{n_0-1}F_{n_1} )\F([n_2+1; n_3,\cdots, n_k])\\				= & F^{(n_1)}\F([n_2; n_3,\cdots, n_k]) + F^{(n_1-1)}\F([n_2+1; n_3,\cdots, n_k])\\				= & (F^{(n_1)}F_{n_2} + F^{(n_1-1)}F_{n_2+1})\F([n_3; n_4,\cdots, n_k])\\				+& (F^{(n_1)}F_{n_2+1} + F^{(n_1-1)}F_{n_2})\F([n_3+1; n_4,\cdots, n_k])\\				=&F^{(n_2)}\F([n_3; n_4,\cdots, n_k]) + F^{(n_2-1)}\F([n_3+1; n_4,\cdots, n_k]),		\end{align*}where $F^{(n_2)} = F^{(n_1)}F_{n_2} + F^{(n_1-1)}F_{n_2+1}$, and $F^{(n_2-1)} = F^{(n_1)}F_{n_2+1} + F^{(n_1-1)}F_{n_2}$. \\If we continue that, we will see that \begin{equation}	F([n_0; n_1, n_2,\cdots, n_k]) = F^{(n_{k-1})}F_{n_k} + F^{(n_{k-1}-1)}F_{n_k+1},\end{equation}for \begin{equation*}	F^{(n_{k-1})} = F^{(n_{k-2})}F_{n_{k-1}} + F^{(n_{k-2}-1)}F_{n_{k-1}+1}.\end{equation*}\end{proof}
		\subsection{Rational-Indexed Lucas Numbers}	In this section, we explain how to write rational-indexed Lucas numbers via the codenominator function $\F$ using the following lemmas.		\begin{lemma}\label{luc:ra}		Let $X=\tfrac{p}{q}$ be a rational number with $\gcd(p,q)=1$, $p>q$, and $q\neq 0$. 		For $X=[3;n]=\tfrac{3n+1}{n}$, the codenominator function $\F([3; n])$ is		\begin{equation*}			\F([3;n]) = L_{n+1},		\end{equation*}		where $L_{n+1}$ is the $(n+1)th$ Lucas number.	\end{lemma}	\begin{proof}		\begin{align*}			\F(X) &= \F\!\left(\tfrac{3n+1}{n}\right) 			= \F\!\left(\tfrac{1}{1+\tfrac{n+1}{n}}\right) \\			&= \F\!\left(1+\tfrac{1}{n}\right) +\ F\!\left(\tfrac{1}{1+\tfrac{1}{n}}\right) \\			&= 2\F(n) + \F(n+1) \\			&= \F(n)+\F(n+2) = L_{n+1},		\end{align*}		since $\F(n)=F_n$ is the $nth$ Fibonacci number. 	\end{proof}		\begin{lemma}			Let $X=[2;1,n]$ and $Y=[3;n]$ for a positive integer $n$. 			Then their codenominator values coincide:			\[			\F([2;1,n]) =\F([3;n]) = L_{n+1}.			\]		\end{lemma}				\begin{proof}			First compute			\begin{align*}				\F([2;1,n]) &= \F\!\left(2+\tfrac{1}{1+\tfrac{1}{n}}\right) 				= \F\!\left(\tfrac{3n+1}{n+1}\right) \\				&= \F\!\left(1+\tfrac{2n+1}{n+1}\right) 				= \F\!\left(\tfrac{1}{1+\tfrac{n+1}{n}}\right) \\				&= 2\F(n)+\F(n+1) = F_n+F_{n+2} = L_{n+1}.			\end{align*}			Similarly,			\begin{align*}				\F([3;n]) &= \F\!\left(3+\tfrac{1}{n}\right) 				= \F\!\left(1+\tfrac{2n+1}{n}\right) \\				&= \F\!\left(\tfrac{1}{1+\tfrac{n+1}{n}}\right) 				= 2\F(n)+\F(n+1) = L_{n+1}.			\end{align*}			Hence $\F([2;1,n])=\F([3;n])=L_{n+1}$.		\end{proof}		\subsection{Inverses of Rational-Indexed Fibonacci and Lucas Numbers}\label{subsec:i}In this section we study inverses of rational indices for Fibonacci and Lucas numbers. The following lemma establishes an explicit formula for the codenominator function at the inverse of a two-term continued fraction.\begin{lemma}\label{l:3}	Let $X^{-1} = [0;n_0,n_1] = \tfrac{1}{n_0+\tfrac{1}{n_1}}$ with $n_0,n_1 \in \mathbb{Z}_{>0}$. 	Let $\F(n)$ denote the codenominator, and let $(F_n)_{n\geq 0}$ be the Fibonacci sequence. Then	\begin{equation}\label{e:l3}		\F([0;n_0,n_1]) = F_{n_0+1}F_{n_1} + F_{n_0}F_{n_1+1}.	\end{equation}\end{lemma}
\begin{proof}	Write $X = [n_0;n_1] = n_0 + \tfrac{1}{n_1}$. Then its inverse is $X^{-1}=[0;n_0,n_1]$. 	Applying the recursive rules \eqref{F:1}–\eqref{F:2} repeatedly, we obtain	\[	\F([0;n_0,n_1]) 	= F_{n_0} \, \F\!\left(\tfrac{1}{1+\tfrac{1}{n_1}}\right) 	+ F_{n_0-1} \, \F\!\left(1+\tfrac{1}{n_1}\right).	\]	Since $\F(\tfrac{1}{1+1/n_1})=\F(n_1) + \F(n_1 + 1)$ and $\F(1+1/n_1)=\F(n_1+1)$, this becomes	\[	\F([0;n_0,n_1]) = F_{n_0}F_{n_1} + F_{n_0-1}F_{n_1}+F_{n_0-1}F_{n_1+1}.	\]	Finally, using $F_{n_0}+F_{n_0-1}=F_{n_0+1}$, we arrive at	\[	\F([0;n_0,n_1]) = F_{n_0+1}F_{n_1} + F_{n_0}F_{n_1+1}.	\qedhere\] \end{proof}

\begin{lemma}\label{l:li}	Let $X=[3;n]$ with $n\in\mathbb{Z}_{>0}$. 	Then the rational-indexed Lucas number at the inverse point $X^{-1}$ satisfies	\begin{equation}\label{l:l1}		L_{X^{-1}} := \F([0;3,n]) = L_{n+1} + F_{n+2} = 2L_{n+1} - F_n\end{equation}	\end{lemma}\begin{proof}By direct calculation, we have	\begin{align*}		L_{X^{-1} }&= \F(X^{-1}) = \F([0; 3, n]) = (L_{n+1})^{-1} = \F(\frac{1}{1+\frac{2n+1}{n}})\\		&= \F(1+\frac{n+1}{n}) + \F(\frac{1}{1+\frac{n+1}{n}}) \\		&= 2\F(\frac{1}{1+\frac{1}{n}}) +\F(1+\frac{1}{n}) = 3\F(n) +2\F(n+1) + \F(n) -\F(n)\\		&= 2(2\F(n) + \F(n+1)) - \F(n) = 2L_{n+1} - F_n = L_{n+1} +F_{n+2}. \qedhere\end{align*}\end{proof}

\section{Inverses of Fibonacci and Lucas Numbers}\label{sec:3}		\subsection{Identities of Fibonacci and Lucas Numbers}				\begin{lemma}\label{id:1}	For every integer $m \geq 1$, the following identity holds:			\[			F_m F_{m+2} = F_{m+1}^2 + (-1)^{m+1}.			\]		\end{lemma}				\begin{proof}			By substituting $n = m+1$ into Cassini’s identity, which is $F_{n+1}F_{n-1} - F_n^2 = (-1)^n$, see (\cite[p.86; Theorem 5.3]{koshy2017fibonacci}). Then we obtain			\[			F_m F_{m+2} - F_{m+1}^2 = (-1)^{m+1},			\]			which proves the claim.		\end{proof}				\begin{lemma}\label{id:2}			For every integer $m \geq 1$, the following identity holds:			\[			F_m L_{m-1} = F_{2m-1} + (-1)^{m+1}.			\]		\end{lemma}				\begin{proof}			Using the Binet formulas for $F_m$ and $L_m$, we compute			\begin{align*}				F_m L_{m-1} &= \frac{\alpha^m - \beta^m}{\alpha-\beta}(\alpha^{m-1} + \beta^{m-1}) \\				&= \frac{\alpha^{2m-1} - \beta^{2m-1}}{\alpha-\beta} 				+ \frac{(\alpha\beta)^m}{\alpha-\beta}\bigl(-(\alpha^{-1} - \beta^{-1})\bigr) \\				&= F_{2m-1} + (-1)^{m+1}F_{-1} \qquad \text{(by Equation~\ref{nf:1})} \\				&= F_{2m-1} + (-1)^{m+1} \qquad \text{(since $F_{-1} = 1$)}.			\end{align*}			Here $\alpha = \tfrac{1+\sqrt{5}}{2}$, $\beta = \tfrac{1-\sqrt{5}}{2}$, and $\alpha\beta = -1$.		\end{proof}				

\begin{lemma}\label{id:3}			For every integer $m \geq 1$, the following identity holds:			\[			L_m \bigl(L_m + F_{m+1}\bigr) = L_{2m} + F_{2m+1} + 3(-1)^m.			\]		\end{lemma}				\begin{proof}			Expanding with Binet formulas gives			\begin{align*}				L_m (L_m + F_{m+1}) 				&= L_m^2 + L_m F_{m+1} \\				&= (\alpha^m + \beta^m)^2 				+ (\alpha^m + \beta^m)\frac{\alpha^{m+1} - \beta^{m+1}}{\alpha - \beta} \\				&= \alpha^{2m} + \beta^{2m} + 2(\alpha\beta)^m 				+ \frac{\alpha^{2m+1} - \beta^{2m+1}}{\alpha-\beta} 				+ \frac{\alpha-\beta}{\alpha-\beta}(\alpha\beta)^m \\				&= L_{2m} + F_{2m+1} + 3(-1)^m,			\end{align*}			since $\alpha\beta = -1$. 			Here $\alpha = \tfrac{1+\sqrt{5}}{2}$ and $\beta = \tfrac{1-\sqrt{5}}{2}$.		\end{proof}			
\subsection{ Inverse of Fibonacci Numbers}	
\begin{proposition}\label{p:f1}	Let $X = [1; n]$ be a continued fraction of length $2$ for some positive integer $n$. 	Then the rational-indexed Fibonacci number satisfies $F_X = F_n$, and the inverse satisfies $F_{X^{-1}} = F_{n+2}$.	\end{proposition}	\begin{proof}	By Lemma \ref{l:4}	\begin{equation*}		F_X = \F([1; n]) = F_1F_n + F_0F_{n+1} = F_n, \qquad \text{since }\qquad F_0 =0, F_1=1.	\end{equation*}		Next, we compute $F_{X^{-1}}$, where $X^{-1} = [0; 1, n]$. By Lemma~\ref{l:3},	\begin{align*}		F_{X^{-1}} & = \F([0; 1, n]) = F_{1+1}F_n + F_1F_{n+1} \qquad \text{(by Equation \ref{e:l3})}\\		&= F_2F_n + F_1F_{n+1} = F_n + F_{n+1}\qquad \text{(since $F_1 = F_2 = 1$)}\\		& = F_{n+2},	\end{align*}	by the definition of Fibonacci numbers.	\end{proof}\begin{proposition}	Let $X = [1; 1, n]$ and $Y = [2; n-2]$. 	Then the associated codenominator functions coincide:	\[	\F(X) = \F(Y).	\]\end{proposition}\begin{proof}	Using the definition and basic properties of the codenominator function $\F$ (Equations~\ref{F:1} and \ref{F:2}), we compute:	\begin{align*}		\F(X) &= \F([1; 1, n]) 		= \F\!\left(1 + \frac{1}{1 + \frac{1}{n}}\right) \\		&= \F\!\left(1 + \frac{n}{n+1}\right) 		= \F\!\left(1 + \frac{1}{n}\right) \\		&= \F(n) = F(n-2) + F(n-1) \\		&= F_2F_{n-2} + F_1F_{n-1} \\		&= \F([2; n-2]) = \F(Y),	\end{align*}	where we used Lemma~\ref{l:2} and the identities $F_1 = F_2 = 1$. 	This proves the claim.\end{proof}
\begin{proposition}\label{p:f2}	Let $X = [2; n-2]$ be a continued fraction of length $2$ for some positive integer $n$. 	Then the rational-indexed Fibonacci number satisfies $F_X = F_n$, and the inverse satisfies $F_{X^{-1}} = L_{n-1}$.\end{proposition}\begin{proof}	By Lemma~\ref{l:2}, we have	\[	F_X = \F([2; n-2]) = F_2F_{n-2} + F_1F_{n-1} = F_n, 	\]	since $F_1 = F_2 = 1$.		Next, we compute $F_{X^{-1}}$, where $X^{-1} = [0; 2, n-2]$. By Lemma~\ref{l:3},	\begin{align*}		F_{X^{-1}} &= \F([0; 2, n-2]) = F_{2+1}F_{n-2} + F_2F_{n-1} \\		&= F_3F_{n-2} + F_2F_{n-1} \\		&= 2F_{n-2} + F_{n-1} \\		&= F_{n-2} + (F_{n-1}+F_{n-2}) \\		&= F_{n} + F_{n-2} \\		&= L_{n-1},	\end{align*}	where we used $F_2 = 1$, $F_3 = 2$, and the identities $F_n = F_{n-1} + F_{n-2}$, $L_n = F_{n-1}+F_{n+1}$.\end{proof}\begin{lemma}\label{in:f1}	Let $X = [1; m]$ for some positive integer $m$. 	Then the rational-indexed Fibonacci number $F_X$ and its inverse $F_{X^{-1}}$ satisfy the following congruences:	\begin{align}		F_{X^{-1}} &\equiv F_{m+1} \pmod{F_X}, \\		F_XF_{X^{-1}} &\equiv (-1)^{m+1} \pmod{F_{m+1}}.	\end{align}\end{lemma}
\begin{proof}	By Proposition~\ref{p:f1}, we know that $F_X = F_m$ and $F_{X^{-1}} = F_{m+2}$. 	From the definition of Fibonacci numbers,	\begin{align*}		F_{X^{-1}} = F_{m+2} &= F_{m+1} + F_{m} \\		&\equiv F_{m+1} \pmod{F_{m}} \\		&\equiv F_{m+1} \pmod{F_X}.	\end{align*}		For the second congruence, we compute	\begin{align*}		F_XF_{X^{-1}} &= F_mF_{m+2} \\		&= F_{m+1}^2 + (-1)^{m+1} \qquad \text{(by Lemma~\ref{id:1})} \\		&\equiv (-1)^{m+1} \pmod{F_{m+1}}. \qedhere	\end{align*}\end{proof}
\begin{lemma}\label{in:f2}		Let $X = [2; m-2]$ for some positive integer $m$. 		Then the rational-indexed Fibonacci number $F_X$ and the inverse $F_{X^{-1}}$ satisfy the following congruence:		\begin{align}			F_{X^{-1}} &\equiv F_{m-2} \pmod{F_X}, \\			F_X F_{X^{-1}} &\equiv (-1)^m \pmod{F_{m-2}}.		\end{align}	\end{lemma}		\begin{proof}		By Proposition~\ref{p:f2}, we have $F_X = F_m$ and $F_{X^{-1}} = L_{m-1}$. 		Using the relation between Lucas and Fibonacci numbers, we obtain		\begin{align*}			F_{X^{-1}} = L_{m-1} &= F_{m} + F_{m-2} \\			&\equiv F_{m-2} \pmod{F_m} \\			&\equiv F_{m-2} \pmod{F_X}.		\end{align*}				We now prove the second congruence. From Proposition~\ref{p:f2} and Lemma~\ref{id:2},		\begin{align*}			F_X F_{X^{-1}} &= F_m L_{m-1} \\			&= F_{2m-1} + (-1)^{m+1}.		\end{align*}				Applying Binet’s formulas (Theorems \ref{b:1}, and Equation \ref{b:2}), we compute		\begin{align*}			F_{2m-1} &= \frac{\alpha^{2m-1}-\beta^{2m-1}}{\alpha-\beta}, 			& F_{m-2} &= \frac{\alpha^{m-2}-\beta^{m-2}}{\alpha-\beta}, \\			\frac{F_{2m-1}}{F_{m-2}} &= \frac{\alpha^{2m-1}-\beta^{2m-1}}{\alpha^{m-2}-\beta^{m-2}}.		\end{align*}		Hence,		\begin{align*}			\alpha^{2m-1}-\beta^{2m-1} 			&= (\alpha^{m+1}+\beta^{m+1})(\alpha^{m-2}-\beta^{m-2}) 			+ (\alpha^3-\beta^3)(-1)^{m-2}.		\end{align*}		Dividing both sides by $\alpha-\beta$ gives		\[		F_{2m-1} = L_{m+1}F_{m-2} + (-1)^{m-2}F_3,		\]		where $\alpha = \tfrac{1+\sqrt{5}}{2}$, $\beta = \tfrac{1-\sqrt{5}}{2}$, and $\alpha\beta = -1$. 				Since $F_3 = 2$, we conclude		\begin{align*}			F_X F_{X^{-1}} &= L_{m+1}F_{m-2} + 2(-1)^{m-2} + (-1)^{m+1} \\			&= L_{m+1}F_{m-2} + (-1)^{m-2}(2+(-1)^3) \\			&= L_{m+1}F_{m-2} + (-1)^{m-2} \\			&\equiv (-1)^m \pmod{F_{m-2}},		\end{align*}		since $(-1)^{m-2} = (-1)^m$. 		This completes the proof.	\end{proof}		\begin{thm}		Let $X$ be a positive rational number and $m$ a positive integer such that $F_X = F_m$. 		Then the inverse of $F_m$ coincides with the inverse of the rational-indexed Fibonacci number $F_{X^{-1}}$. In particular,		\[		(F_m)^{-1} :=		\begin{cases}			F_{m+2}, & \text{if } F_m = \F([1; m]), \\[6pt]			L_{m-1}, & \text{if } F_m = \F([2; m-2]).		\end{cases}		\]	\end{thm}		\begin{proof}		By Lemmas~\ref{in:f1}, \ref{in:f2}, \ref{id:1}, and \ref{id:2}, 		the product $F_X F_{X^{-1}}$ is always an identity. 		Therefore, if $F_X = F_m$ for some $m \in \mathbb{Z}_{+}$, it follows that		$F_{X^{-1}} = (F_m)^{-1}$. 	\end{proof}		\subsection{ Inverse of Lucas Numbers}\begin{lemma}\label{in:l}	Let $X = [1; m]$ be a positive rational index for some $m \in \mathbb{Z}_{>0}$. Then the rational-indexed Lucas number $L_X$ and its inverse $L_{X^{-1}}$ satisfy the following congruence:	\begin{align}		L_{X^{-1}} &\equiv F_{m+1} \pmod{L_X}, \\		L_X L_{X^{-1}} &\equiv (-1)^{m+1} \pmod{F_{m+1}}.	\end{align}\end{lemma}
\begin{proof}	By the definitions of Lucas and Fibonacci numbers, we have	\begin{align*}		L_{X^{-1}} &= L_m + F_{m+1} \qquad \text{(by Lemma~\ref{l:li})} \\		&\equiv F_{m+1} \pmod{L_m} \qquad \text{(since $L_X = L_m$ for $m \in \mathbb{Z}_{>0}$)} \\		&\equiv F_{m+1} \pmod{F_X}.	\end{align*}		Now, let us prove the second congruence. By Lemma~\ref{id:3},	\begin{align*}		L_X L_{X^{-1}} &= L_m(L_m + F_{m+1}) \\		&= L_{2m} + F_{2m+1} + 3(-1)^m.	\end{align*}		From the Binet formulas (Theorems \ref{b:2}, and \ref{b:2}), namely	\[	L_{2m} = \alpha^{2m} + \beta^{2m}, 	\qquad 	F_{n} = \frac{\alpha^n - \beta^n}{\alpha - \beta},	\]	with $\alpha = \tfrac{1+\sqrt{5}}{2}$, $\beta = \tfrac{1-\sqrt{5}}{2}$ and $\alpha\beta = -1$, one obtains	\begin{align*}		L_{2m} &= 5F_{m+1}F_{m-1} + 3(-1)^{m-1}, \\		F_{2m+1} &= L_m F_{m+1} + (-1)^{m+1}.	\end{align*}		Hence,	\begin{align*}		L_X L_{X^{-1}} &= (5F_{m-1} + L_m)F_{m+1} + (-1)^{m+1} \\		&\equiv (-1)^{m+1} \pmod{F_{m+1}}.	\qedhere\end{align*}\end{proof}
\begin{thm}		Let $X$ be a positive rational number and $m \in \mathbb{Z}_{>0}$ such that the rational-indexed Lucas number satisfies $L_X = L_m$. Then the inverse of $L_m$ coincides with the inverse of the rational-indexed Lucas number $L_{X^{-1}}$. In particular,		\begin{equation}			(L_m)^{-1} := (L_{[3; m-1]})^{-1} = L_m + F_{m+1}.		\end{equation}	\end{thm}		\begin{proof}We know from Lemma~\ref{id:3} and Lemma~\ref{in:l} that the product of $L_X$ and $L_{X^{-1}}$ is always the identity. Hence, if $L_X = L_m$ for some positive integer $m$, it follows that\[L_{X^{-1}} = (L_m)^{-1}.\qedhere\]\end{proof}	\subsection{Examples}\begin{example}	What are the inverses of the $11^{\text{th}}$ Fibonacci and Lucas numbers? The answer can be obtained directly using Lemma~\ref{l:3} and Lemma~\ref{l:li}.\begin{equation*}(F_{11})^{-1} =\begin{cases}\F([0;2,9]) = L_{10} = 123,\\\F([0;1,11]) = F_{13} = 233.\end{cases}\end{equation*}	Thus,		\begin{equation*}			(F_{11})^{-1} = 233 \equiv F_{12} = 144 \pmod{F_{11}},		\end{equation*}		then		\begin{equation*}			 F_{11}(F_{11})^{-1} = F_{11}F_{13} = 89 \cdot 233 \equiv (-1)^{12} = 1 \pmod{F_{12}}.			 \end{equation*} 			and 			 \begin{equation*}			 (F_{11})^{-1} = 123 \equiv F_9 = 34 \pmod{F_{11}},			 \end{equation*}			 then			 \begin{equation*}			 	 F_{11}(F_{11})^{-1} = F_{11}L_{10} = 89 \cdot 123 \equiv (-1)^{11} = -1 \pmod{F_9}.			 \end{equation*}		\begin{equation*}		(L_{11})^{-1} = \F([0; 3, 11-1]) = \F([0; 3, 10]) = L_{11} + F_{12} = 343, \qquad \text{by Lemma~\ref{l:li}}	\end{equation*}	Thus, 	 \begin{equation*}		(L_{11})^{-1} = 343 \equiv F_{12} = 144 \pmod{L_{11}},	 \end{equation*}	then 	\begin{equation*}		L_{11}(L_{11})^{-1} = 199 \cdot 343 \equiv (-1)^{12} = 1 \pmod{F_{12}}.	\end{equation*}	\end{example} \section{Results}The following corollary specifies the conditions under which rational-indexed Fibonacci numbers differ from the classical ones.\begin{corollary}	Let $X$ be a positive rational number with continued fraction expansion	\[	X = \frac{p}{q} = [n_0; n_1, \ldots, n_k],	\]	where $\gcd(p,q)=1$ and $q \neq 0$. Under the following conditions, there exist infinitely many rational-indexed Fibonacci numbers that are distinct from the classical Fibonacci numbers:	\begin{enumerate}		\item \label{c:f1} If $k=1$ and $n_0 \geq 3$, then the rational-indexed Fibonacci number differs from the classical one. In addition, if $n_0 = 3$ for $k=1$, then the codenominator function $\F([3; n_1])$ is the Lucas number.		\item If $k=2$ and $n_0 \geq 2$, then the rational-indexed Fibonacci number differs from the classical one.		\item If $k \geq 3$, then the rational-indexed Fibonacci number always differs from the classical one.	\end{enumerate}\end{corollary}\begin{proof}	\begin{enumerate}		\item If $k=1$ and $n_0 \geq 3$, then $X = [n_0; n_1]$. By Lemma~\ref{l:2},		\[		F_X = \F([n_0; n_1]) = F_{n_0}F_{n_1} + F_{n_0-1}F_{n_1+1}.		\]		If $n_0=1$, then		\[		F_X = \F([1; n_1]) = F_1F_{n_1} + F_0F_{n_1+1} = F_{n_1}, \qquad (F_0=0,\, F_1=1).		\]		If $n_0=2$, then		\[		F_X =\F([2; n_1]) = F_2F_{n_1} + F_1F_{n_1+1} = F_{n_1+2}.		\]		Thus, for $n_0 \leq 2$, the rational-indexed Fibonacci number $F_X$ coincides with a classical Fibonacci number. Since $F_{n_0} \geq 2$ for $n_0 \geq 3$, in this case $F_X$ always differs from the classical ones.\\		If we take $n_0=3$ for $k=1$, then we know that $\F([3; n_1]) = L_{n+1}$ by Lemma \ref{luc:ra}.				\item The case $k=2$ follows similarly. For $n_0=1$, by Lemma~\ref{l:4},		\begin{align*}			F_X &= \F([1; n_1, n_2]) \\			&= (F_{n_1-1}L_1 + F_{n_1-2}F_2)F_{n_2} + (F_{n_1-2}L_1 + F_{n_1-3}F_2)F_{n_2+1} \\			&= F_{n_1}F_{n_2} + F_{n_1-1}F_{n_2+1} \qquad (\text{by Lemma~\ref{l:2}}) \\			&= \F([n_1; n_2]).		\end{align*}		Hence, for $k=2$ and $n_0=1$, if $n_1 \geq 3$, the rational-indexed Fibonacci number is different from the classical ones by part~\ref{c:f1}. For $n_0 \geq 2$, Lemma~\ref{l:4} implies that $F_X$ is always distinct, since $L_{n_0} \geq 3$ and $F_{n_0} \geq 2$.		\item For $k \geq 3$, Lemmas~\ref{l:6} and~\ref{l:7} show that $F_X = F_{[n_0; n_1,\ldots,n_k]}$ is always distinct from the classical Fibonacci numbers.	\end{enumerate}	This completes the proof.\end{proof}\begin{corollary}	For $m \in \mathbb{Z}_+$, the following identities hold:	\begin{enumerate}		\item For $X = [1; m]$,		\[		F_m(F_m)^{-1} := F_mF_{m+2} = F_{m+1}^2 + (-1)^{m+1} \equiv (-1)^{m+1} \pmod{F_{m+1}}.		\]		\item For $X = [2; m-2]$,		\[		F_m(F_m)^{-1} := F_mL_{m-1} = L_{m+1}F_{m-2} + (-1)^m \equiv (-1)^m \pmod{F_{m-2}}.		\]		\item For $X = [3; m-1]$,		\[		L_m(L_m)^{-1} := L_m(L_m + F_{m+1}) = (5F_{m-1} + L_m)F_{m+1} + (-1)^{m+1} \equiv (-1)^{m+1} \pmod{F_{m+1}}.		\]	\end{enumerate}\end{corollary}\begin{proof}	The proof follows directly from the established identities.\end{proof}

\section{Conclusion}In this paper, we extended the framework of rational-indexed Fibonacci numbers via the codenominator function $\F$. We derived explicit formulas for indices with continued fraction expansions of various lengths and, more generally, established a closed-form expression valid for arbitrary rational indices. We also determined precise conditions under which these numbers coincide with Lucas numbers or differ from classical Fibonacci numbers. Moreover, we showed that every Fibonacci and Lucas number admits a multiplicative inverse within this framework. These results highlight new structural properties of Fibonacci- and Lucas-related sequences and open further directions for investigation in number theory.

\section*{Acknowledgements}We are grateful to the anonymous referee for the thoughtful and constructive suggestions that greatly enhanced the clarity and presentation of this paper. We also thank the editors for their guidance and support during the review process.

\section*{Disclosure statement}No conflict of interests was reported by the authors.

%%%%%%%%%%%%%%%%%%%%%%%%%%%%%%%%%%%%%%%%%%

\medskip%AMS classifications available at https://mathscinet.ams.org/mathscinet/freetools/msc-search\noindent MSC2020: 11B39, 11A55
%%%%%%%%%%%%%%%%%%%%%%%%%%%%%%%%%%%%%%%%%
\end{document}